\documentclass[final,leqno]{siamltex}
\usepackage{graphicx}

\title{On the Optimal General Convergence  Rates for Quadratures Derived from Chebyshev Points\thanks{This work was
supported by National Science Foundation of China (No. 11371376)}}

\author{Shuhuang Xiang\thanks{Department of Applied
Mathematics and Software, Central South University, Changsha, Hunan
410083, P. R. China. Email: xiangsh@mail.csu.edu.cn.}}

\begin{document}
\maketitle

\begin{abstract}
In this paper, we study the optimal general convergence rates for quadratures derived from Chebyshev points. By building on the aliasing errors on
integration of Chebyshev polynomials, together with the asymptotic formulae on
the coefficients of Chebyshev expansions, new and optimal convergence rates for
$n$-point Clenshaw-Curtis, Fej\'{e}r's first and second quadrature rules are
established for Jacobi weights or Jacobi weights multiplied by $\ln((x+1)/2)$. The convergence orders are attainable for some functions of finite regularities.  In addition, by using  refined estimates on aliasing errors on integration of Chebyshev polynomials by Gauss-Legendre  quadrature, an improved convergence rate for Gauss-Legendre is given too.
\end{abstract}



\pagestyle{myheadings}
\thispagestyle{plain}
\markboth{Optimal Rates}{Gauss-type Quadrature and Quadratures derived from Chebyshev Points}

\section{Introduction}

The computation of integrals of the form of
\begin{equation}
I[f]=\int_{-1}^1w(x)f(x)dx
\end{equation}
is one of the oldest and most important issues in numerical analysis.
Quadrature formulae are usually derived from polynomial
interpolation by a finite sum
\begin{equation}{\displaystyle
I_n[f]=\sum_{j=1}^nw_jf(x_j),\quad x_j\in [-1,1]}.
\end{equation}

Among all interpolation type quadrature rules with $n$ nodes,
the Gauss-Christoffel formula, denoted by $I_n^G[f]$, has the highest accuracy of degree $2n-1$
(c.f. Davis and Rabinowitz \cite{Davis}, Gautschi \cite{Gautschi0}). Particularly, for Jacobi weight function $w(x)=(1-x)^{\alpha}(1+x)^{\beta}$ ($\alpha>-1$, $\beta>-1$), fast evaluation of the nodes and weights for the Gauss quadrature was
given by Glaser, Liu and Rokhlin \cite{Glaser} with $O(n)$ operations, which has been recently extended by both Bogaert, Michiels and Fostier
\cite{Bogaert}, and Hale and Townsend \cite{Hale}. A {\sc Matlab} file for computation of these nodes and weights can be found in {\sc Chebfun}
system \cite{Trefethen3}.

It has been observed for a long time that, in the case $w(x)\equiv 1$, for most integrands, $n$-point Gauss
and $n$-point Clenshaw-Curtis quadrature (denoted by $I_{n}^{C\texttt{-}C}[f]$) are about equally accurate (c.f. O'Hara
and Smith \cite{Hara}, Evans \cite{Evans} and Kythe and
Sch\"{a}ferkotter \cite{Kythe}. For more details, see Trefethen \cite{Trefethen1}).

This observation was made precise by Trefethen \cite{Trefethen1,Trefethen2}, by using new asymptotics on
the coefficients of Chebyshev expansions for functions of finite
regularity:
Suppose $f(x)$ satisfies a
Dini-Lipschitz condition on $[-1,1]$, then it has the following
uniformly convergent Chebyshev series expansion (c.f. Cheney \cite[p.
129]{Cheney})
\begin{equation}
f(x)=\sum_{j=0}^{\infty}{'}a_jT_j(x),
\end{equation}
where the prime denotes summation whose first term is halved,
$T_j(x)=\cos(j\cos^{-1}x)$ denotes the Chebyshev polynomial of
degree $j$, and the Chebyshev coefficient $a_j$ is defined by
\begin{equation}a_j=\frac{2}{\pi}\int_{-1}^{1}\frac{f(x)T_j(x)}{\sqrt{1-x^2}}dx,\quad
j=0,1,\ldots.
\end{equation}

Trefethen in \cite{Trefethen1,Trefethen2}\footnote{In \cite{Trefethen1,Trefethen2}, the quadrature error bound is considered for $(n+1)$-point Gauss and Clenshaw-Curtis quadrature.} showed that for an integer $k\ge 1$, if $f(x)$ has an
absolutely continuous $(k-1)$st derivative $f^{(k-1)}$ on $[-1,1]$
and a $k$th derivative $f^{(k)}$ of bounded variation
$V_k={\rm Var}(f^{(k)})<\infty$, then for each $j\ge k+1$,
\begin{equation}
|a_j|\le{\displaystyle\frac{2V_k}{\pi j(j-1)\cdots(j-k)}},
\end{equation}
and
\begin{equation}
{\frac{32V_k}{15k\pi
(2n-1-k)^k}\ge}\left\{\begin{array}{ll}|I[f]-I_{n}^G[f]|&\mbox{for all $n\ge k/2+1$}\\
|I[f]-I_n^{C\texttt{-}C}[f]|&\mbox{for all sufficiently large $n$}
\end{array}.\right.
\end{equation}

Chebyshev expansions are very useful tools for numerical analysis. Their convergence is guaranteed under rather general conditions and they often converge fast compared with other polynomial expansions (c.f. Fox and Parker \cite{Fox}, Hesthaven et al. \cite{Hesthaven}, Petras \cite{Petras} and Xiang \cite{Xiang1}). For example, it has been shown that
the coefficient $a_j$ of the Chebyshev expansion of $f$ decays a
factor of $\sqrt{j}$ faster than the corresponding coefficient of
the Legendre expansion, which is mentioned in
\cite[p. 17]{Fox} and Boyd \cite[p. 52]{Boyd}, and made precise in  \cite{Xiang1} and Wang and Xiang \cite{WangXiang}. Additionally, the quadrature errors of the Gauss and Clenshaw-Curtis can be represented by using the Chebyshev expansion, respectively, if $\sum_{j=1}^{\infty}a_j$ is absolutely convergent, as
\begin{equation}\quad\,\,\mbox{\small
${\displaystyle E_n^G[f]=I[f]-I_n^G[f]=\sum_{j=2n}^{\infty}a_jE_n^G[T_j],\, E_n^{C\texttt{-}C}[f]=I[f]-I_n^{C\texttt{-}C}[f]=\sum_{j=n}^{\infty}a_jE_n^{C\texttt{-}C}[T_j]}$.}
\end{equation}

A new convergence rate improved one further power of $n$ for
$n$-point Gauss and Clenshaw-Curtis  quadrature is
given in Xiang and Bornemann \cite{XiangBornemann}  for $f\in X^s$ ($s>0$), based on the
work of Curtis and Rabinowitz \cite{Curtis} and  Riess and Johnson \cite{Riess} from the early 1970s, and a refined estimate for
Gauss quadrature applied to Chebyshev polynomials due to Petras in 1995 \cite{Petras}. Here, we say $f\in X^s$ if the Chebyshev coefficient $a_j$ satisfies that $a_j=O(j^{-s-1})$ \cite{XiangBornemann}. Moreover, from \cite{Trefethen1,Trefethen2}, we see that if $f(x)$ has an
absolutely continuous $(k-1)$st derivative $f^{(k-1)}$ on $[-1,1]$
(if $k\ge 1$) and $V_k<\infty$ then $f\in X^k$. %The convergence rate is also extended to Fej\'{e}r's quadrature rules in \cite{Xiang2}.

In this paper, along the way to  \cite{Curtis,Riess,Trefethen1,Trefethen2,XiangBornemann},  by using refined estimates on the aliasing errors about the integration of Chebyshev polynomials by Gauss quadrature, in Section 2, we will improve the convergence rate for $n$-point Gauss-Legendre quadrature for $f\in X^s$ as
$$
E_n^G[f]=\left\{\begin{array}{ll}
O(n^{-2s}),&0 < s <1\\
O(n^{-2}\ln n),&s=1\\
O(n^{-s-1}),& s>1
\end{array}.\right.
$$
In Section 3, we will present optimal general convergence rates for generalized $n$-point Clenshaw-Curtis quadrature, Fej\'{e}r's first and second rules for $f\in X^s$ for the following weights:
\begin{itemize}
\item for $w(x)=(1-x)^{\alpha}(1+x)^{\beta}$:
$$
E_n[f]=\left\{\begin{array}{ll}
O(n^{-s-1})&\mbox{if $\min(\alpha,\beta)\ge -\frac{1}{2}$}\\
O(n^{-s-2-2\min(\alpha,\beta)})&\mbox{if $-1<\min(\alpha,\beta)< -\frac{1}{2}$}\end{array},\right.
$$
\item for $w(x)=\ln((x+1)/2)(1-x)^{\alpha}(1+x)^{\beta}$:
$$
E_n[f]=\left\{\begin{array}{ll}
O(n^{-s-1})&\mbox{if $\beta> -\frac{1}{2}$}\\
O(n^{-s-2-2\beta}\ln n)&\mbox{if $-1<\beta\le -\frac{1}{2}$}\end{array}.\right.
$$
\end{itemize}
Without ambiguity, here $E_n[f]$ denotes the quadrature error of the $n$-point  Clenshaw-Curtis quadrature, Fej\'{e}r's first and second rules for function $f\in X^s$, respectively. It is worth noting that these convergence orders are attainable for some functions of finite regularities. Final remarks on comparison with the convergence rate of Gauss quadrature is included in Section 4.

\section{An improved error bound on the Gauss quadrature for $w(x)\equiv 1$} Let $x_k$ be the zeros of
the Legendre polynomial of degree $n$, ordered by $-1<x_1<x_2<\cdots<x_n<1$, and $w_k$ the corresponding weights in the $n$-point Gauss quadrature ($k=1,2,\ldots,n$).

Xiang and Bornemann \cite{XiangBornemann}  showed for $f\in X^s$ that
\begin{equation}\label{gausserror}
E_n^G[f]=\left\{\begin{array}{ll}
O(n^{-3s/2}),&0 < s <2\\
O(n^{-s-1}),& s\ge 2
\end{array},\right. \mbox{while  $E_n^{C\texttt{-}C}[f]=O(n^{-s-1})$ for $s>0$},
\end{equation}
by applying the asymptotic formulae for $x_k=-\cos\theta_k$ and $w_k$  for $1\le k\le \lfloor (n+1)/2\rfloor$\footnote{Here, $\lfloor (n+1)/2\rfloor$ denotes the integral part of $(n+1)/2$.}
\begin{eqnarray}\,\,\,\,\,\,
\theta_k&=&\phi_k+\frac{1}{2(2n+1)^2}\cot\phi_k+\delta_k, \quad \phi_k=\frac{4k-1}{4n+2}\pi \quad \mbox{(c.f. Gatteschi \cite{Gatteschi87})},\\
\,\,\,\,\,\,w_k&=&\frac{2\pi}{2n+1}\sin\phi_k\left(1-\frac{1}{2(2n+1)^2}\right)(1+\epsilon_k)\quad \mbox{(c.f. F\"{o}rster and Petras \cite{ForsterPetras93})},
\end{eqnarray}
where
\begin{equation}\label{asynodeweight}
\quad 0\le -\delta_k\le \frac{11\cos\phi_k}{8(2n+1)^4\sin^3\phi_k}, \quad -\frac{\cos^2\phi_k}{(2n+1)^4\sin^4\phi_k}\le \epsilon_k\le \frac{8}{(2n+1)^4\sin^4\phi_k},
\end{equation}
together with the error estimate given by Petras \cite{Petras} for $m=j(4n+2)+2r$
$$
|E_n^G[T_m]|=\left\{\begin{array}{ll}
\frac{2+O(mr/n^2)}{|4r^2-1|}+O(m^4/n^6) + O(m^2\log n/n^2),&-n < r < n\\
\frac{\pi}{2}+ O(m/n^2) +O(m^4/n^6)+O(m\log n/n^2),& r=\pm n
\end{array}.\right.
$$

By using the following refined estimates, we can get an improved convergence rate on the Gauss quadrature.

\begin{lemma} The aliasing and aliasing errors about the integration of Chebyshev polynomials by the $n$-point Gauss quadrature satisfy that for  $j\ge 1$,
\begin{eqnarray}\label{gausscheby}
\quad\quad  I_n^G[T_m]&=&\left\{\begin{array}{ll}(-1)^j\frac{2}{1-4r^2}+O(m/n^2),&\ m=j(4n+2)+2r, \quad -n < r < n\\
\pm(-1)^j\frac{\pi}{2}+O(m/n^2),&\  m=(2j-1)(2n+1)\pm 1\end{array},\right.\\
\quad\quad \quad  |E_n^G[T_m]|&=&\left\{\begin{array}{ll} \frac{2}{|4r^2-1|}+O(m/n^2),&\ m=j(4n+2)+2r,\quad -n < r < n\\
\frac{\pi}{2}+O(m/n^2),&\  m=(2j-1)(2n+1)\pm 1\end{array}.\right.
\end{eqnarray}
\end{lemma}
\begin{proof}
For the case $m=j(4n+2)+2r$ with $-n<r<n$ and $j\ge 1$: From (2.1), we have $$m\theta_k=j(4n+2)\theta_k+2r\theta_k=j(4k-1)\pi+h_k+2r\theta_k,$$ where
\begin{equation}\label{asyonhk}
h_k=\frac{j(4n+2)}{2(2n+1)^2}\cot\phi_k+j(4n+2)\delta_k,
\end{equation}
and get
$$\begin{array}{lll}
\cos m\theta_k&=&\cos (j(4k-1)\pi+h_k+2r\theta_k)\\
&=&(-1)^j\cos h_k\cos 2r\theta_k-(-1)^j\sin h_k \sin 2r\theta_k\\
&=&(-1)^j(1+\cos h_k-1)\cos 2r\theta_k-(-1)^j\sin h_k \sin 2r\theta_k\\
&=&(-1)^j\cos 2r\theta_k-2(-1)^j\sin \frac{h_k}{2}\sin\left(\frac{h_k}{2}+2r\theta_k\right),\end{array}
$$
which yields
\begin{equation}\mbox{\quad \quad\small ${\displaystyle
I_n^G[T_m]=\sum_{k=1}^n w_k\cos m\theta_k=(-1)^j
I[T_{|2r|}]-2(-1)^j\sum_{k=1}^n w_k \sin \frac{h_k}{2}\sin\left(\frac{h_k}{2}+2r\theta_k\right).}$}
\end{equation}
Furthermore, note that
\begin{equation}{\small\begin{array}{lll}
{\displaystyle\Big|\sum_{k=1}^n w_k \sin \frac{h_k}{2} \sin\left(\frac{h_k}{2}+2r\theta_k\right)\Big|}
&\le&{\displaystyle \sum_{k=1}^{\lfloor (n+1)/2\rfloor} w_k |h_k|}\\
&\le&{\displaystyle \sum_{k=1}^{\lfloor (n+1)/2\rfloor} w_k\left( \frac{j(4n+2)}{2(2n+1)^2}\cot\phi_k+j(4n+2)|\delta_k|\right).}\end{array}}
\end{equation}
From the estimate on $\delta_k$ (2.3) and using $\frac{2}{\pi}\phi_k\le \sin \phi_k\le \phi_k$ for $0<\phi_k\le \frac{\pi}{2}$, we obtain $$\delta_k=O(n^{-1}k^{-3}),$$ and applying an $O(n^{-1})$ bound on the weights from (2.2) or Szeg\"{o} \cite{Szego}, we obtain
\begin{equation}
\sum_{k=1}^{\lfloor (n+1)/2\rfloor} w_k j(4n+2)|\delta_k|=O(m/n^2)\sum_{k=1}^{\lfloor (n+1)/2\rfloor} \frac{1}{k^3}=O(m/n^2).
\end{equation}
Moreover, by (2.2) and (2.3), it is easy to  derive that $$w_k=\frac{2\pi}{2n+1}\sin\phi_k+O(n^{-2}k^{-1}),$$
which, together with the estimate $\cot \phi_k\le \frac{1}{\sin\phi_k}\le \frac{1}{\frac{2}{\pi} \phi_k}=\frac{(2n+1)}{4k-1}$, induces
\begin{equation}\begin{array}{lll}
&&{\displaystyle\sum_{k=1}^{\lfloor (n+1)/2\rfloor} w_k \frac{j(4n+2)\pi}{2(2n+1)^2}\cot \phi_k}\\
&=&{\displaystyle\sum_{k=1}^{\lfloor (n+1)/2\rfloor} \frac{j(4n+2)\pi^2}{(2n+1)^3}\cos \phi_k+\sum_{k=1}^{\lfloor (n+1)/2\rfloor}\frac{j(4n+2)O(n^{-2}k^{-1})}{2(2n+1)^2}\cot \phi_k}\\
&=&{\displaystyle O(m/n^2)\int_{0}^{\pi/2}\cos t dt+O(m/n^3)\sum_{k=1}^{\lfloor (n+1)/2\rfloor} \frac{1}{k^2}}\\
&=&{\displaystyle O(m/n^2)}.\end{array}
\end{equation}
Combining (2.9)-(2.11) derives ${\displaystyle\sum_{k=1}^n w_k \sin \frac{h_k}{2}\sin\left(\frac{h_k}{2}+2r\theta_k\right)=O(m/n^2)}$. Consequently, by (2.8)  we get (2.5), and then using $I[T_{2\ell}]=\frac{2}{1-4\ell^2}$ for $\ell\ge 0$ we get (2.6), in the case  $m=j(4n+2)+2r$ with $-n<r<n$ and $j\ge 1$.

For the case $m=(2j-1)(2n+1)\pm 1$: $\cos m\theta_k$ can be written by (2.1) as
$$\begin{array}{lll}
\cos m\theta_k&=&\cos ((2j-1)(2n+1)\theta_k\pm \theta_k)\\
&=&\cos\left(\frac{(2j-1)(4k-1)}{2}\pi+\widetilde{h}_k\pm \theta_k\right)\\
&=&(-1)^{j+1}\sin(\widetilde{h}_k\pm \theta_k)\\
&=&(-1)^{j+1}\sin \widetilde{h}_k\cos \theta_k\pm (-1)^{j+1}[1+(\cos(\widetilde{h}_k)-1)]\sin \theta_k\\
&=&\pm (-1)^{j+1}\sin \theta_k + 2(-1)^{j+1}\sin\frac{\widetilde{h}_k}{2}\cos \left(\frac{\widetilde{h}_k}{2}\pm \theta_k\right),\end{array}
$$
where  $$\widetilde{h}_k=\frac{(2j-1)(2n+1)}{2(2n+1)^2}\cot\phi_k+(2j-1)(2n+1)\delta_k.$$
By the same arguments as those for the estimate of $\sum_{k=1}^nw_k|h_k|$, similarly, we have $\sum_{k=1}^nw_k|\widetilde{h}_k|=O(m/n^{2})$ and then
\begin{equation}\begin{array}{lll}
I_n^G[T_m]&=&{\displaystyle\pm (-1)^{j+1}\sum_{k=1}^nw_k\sin \theta_k+2(-1)^{j+1}\sum_{k=1}^nw_k\sin\frac{\widetilde{h}_k}{2}\cos \left(\frac{\widetilde{h}_k}{2}\pm \theta_k\right)}\\
%&=&{\displaystyle\pm (-1)^{j+1}\sum_{k=1}^nw_k\sin \theta_k+O(m/n^{2})}\\
&=&{\displaystyle\pm  (-1)^{j+1}\sum_{k=1}^nw_k\sin (-\cos^{-1} x_k)+O(m/n^{2})}\\
&=&\pm  (-1)^{j}I_n^G[\sqrt{1-x^2}] +O(m/n^{2}).\end{array}
\end{equation}
Furthermore, from  F\"{o}rster and Petras \cite{ForsterPetras}, we find that
$$|I_n^G[\sqrt{1-x^2}]-I[\sqrt{1-x^2}]|=|2(I[g(x)]-I_n^G[g(x)])|\le 2\sin^2\frac{2\pi}{(2n+1)^2}$$
by setting $g(x)=-\frac{1}{2}\sqrt{1-x^2}$  and applying the fact that $g$ is convex on $[-1,1]$ with $g(-1)-2g(0)+g(1)=1$, %and then $I_n^G[\sqrt{1-x^2}]-I[\sqrt{1-x^2}]=O(1/n^2)$,
which, together with $I[\sqrt{1-x^2}]=\frac{\pi}{2}$, $I[T_{2\ell}]=\frac{2}{1-4\ell^2}$ for $\ell\ge 0$ and (2.12), derives the desired results in the case $m=(2j-1)(2n+1)\pm 1$.
\end{proof}

\begin{theorem}
If $f \in X^s$, the error of the $n$-point Gauss quadrature has the rate
\begin{equation}\label{gaussnorder2}
E_n^G(f)=\left\{\begin{array}{ll}
O(n^{-2s}),&0 < s <1\\
O(n^{-2}\ln n),&s=1\\
O(n^{-s-1}),& s>1
\end{array}.\right.
\end{equation}
\end{theorem}
\begin{proof}
With $f\in X^s$, that is, $a_m = O(m^{-s-1})$ for some $s > 0$, we see that $${\displaystyle E^{G}_n [f]= \sum_{m=2n}^{\infty}a_{m}E^{G}_n [T_{m}]}$$ is uniformly and absolutely convergent since $a_m = O(m^{-s-1})$ and $|E_n^G[T_m]|\le\frac{32}{15}$ for $m\ge 4$ (c.f. Brass and Petras
\cite{Brass}). Then $E_n^G[f]$ can be estimated, by the asymptotics on $a_m$, estimates (2.6) on $|E_n^G[T_m]|$ and using $E_n^G[T_{2k+1}]=0$ for $k=0,1,\ldots$, as

$$\begin{array}{lll}
\Big|E_n^G(f)\Big|&=&{\displaystyle\Big|\sum_{j=1}^{n}\sum_{|r|<n}a_{j(4n+2)+2r}E_n^G(T_{j(4n+2)+2r})}\\
&&{\displaystyle +\sum_{j=1}^{n}a_{(2j-1)(2n+1)\pm 1}E_n^G(T_{(2j-1)(2n+1)\pm 1})+\sum_{m=4n(n+1)}^{\infty}a_mE_n^G[T_m]}\Big|\\
&=& O\left({\displaystyle  \sum_{j=1}^{n}\sum_{|r|<n}\frac{2/|4r^2-1|}{(j(4n+2)+2r)^{1+s}}+\sum_{j=1}^{n}\frac{\pi}{((2j-1)(2n+1)\pm 1)^{1+s}}}\right)\\
&&{\displaystyle  +O\left(\frac{1}{n^2}\right)\sum_{m=2n}^{4n(n+1)-1}\frac{1}{m^{s}} +O\left(\sum_{m=4n(n+1)}^{\infty}\frac{|E_n^G[T_m]|}{m^{1+s}}\right)}\\
&=&{\displaystyle O\left(\frac{1}{n^{1+s}}\right)+O\left(\frac{1}{n^2}\right)\sum_{m=2n}^{4n(n+1)-1}\frac{1}{m^{s}}+O\left(\sum_{m=4n(n+1)}^{\infty}\frac{1}{m^{1+s}}\right)}\\
&=&{\displaystyle O\left(\frac{1}{n^{1+s}}\right)+O\left(\frac{1}{n^2}\int_{2n}^{4n(n+1)-1}x^{-s}dx\right)+O\left(\frac{1}{n^{2s}}\right)},
\end{array}
$$
which leads to the desired result based up $0<s<1$, $s=1$ and $s>1$, respectively.
\end{proof}

{\sc Remark 1}. {\it   The convergence rate (2.13) is optimal for $s>1$, which is verified similarly with $f_s(x) = |x-0.3|^s\in X^s$ used in \cite{XiangBornemann} (see the right two columns in Figures 2.1-2.2, respectively). While for $f\in X^s$ with $0 < s \le 1$, the convergence rate (2.13) is better than that in \cite{XiangBornemann}. However, the numerical examples in  \cite{XiangBornemann} show that the $n$-point Gauss quadrature also enjoys the  same convergence rate $O(n^{-s-1})$ (see the left column in Figures 2.1-2.2, respectively). }

\begin{figure}[hpbt]
\centerline{\includegraphics[height=6cm,width=15cm]{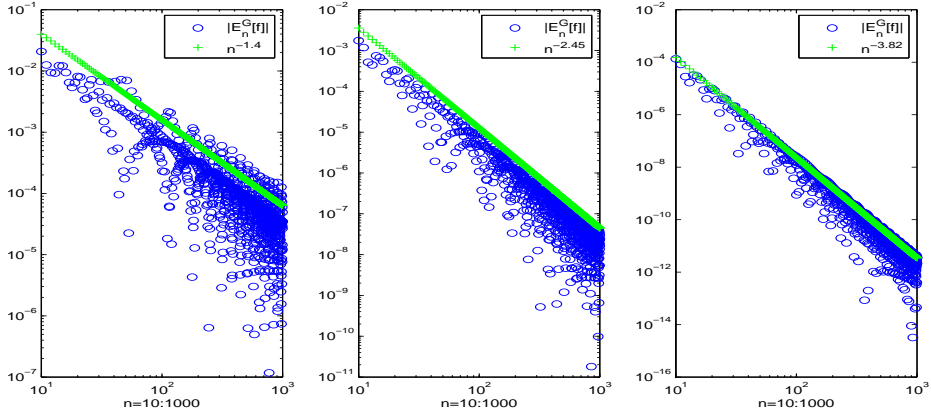}}
\caption{The absolute errors for $n$-point Gauss for $f(x)=|x-0.3|^s$ ($f\in X^s$) with $s=0.4,1.45, 2.82$, respectively: $n=10:1000$.}
\end{figure}

\begin{figure}[hpbt]
\centerline{\includegraphics[height=7cm,width=15cm]{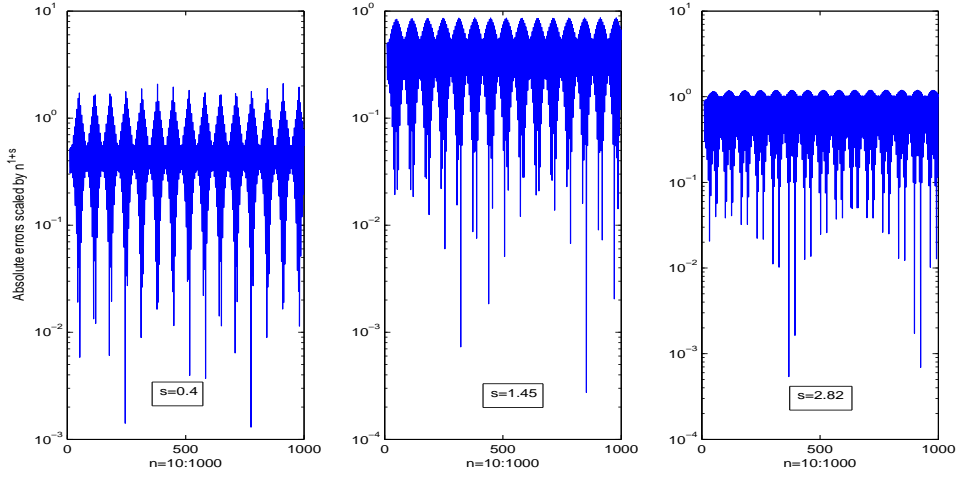}}
\caption{The absolute errors scaled by $n^{1+s}$ for $n$-point Gauss for $f(x)=|x-0.3|^s$ ($f\in X^s$) with $s=0.4,1.45, 2.82$, respectively: $n=10:1000$.}
\end{figure}

{\sc Remark 2}. {\it These techniques are difficult to be extended to study Gauss-Christoffel quadrature for general Jacobi weight functions.
However, following the ideas of Riess and Johnson \cite{Riess}, Trefethen \cite{Trefethen1,Trefethen2} and Xiang and Bornemann \cite{XiangBornemann}, the optimal general convergence rates for generalized $n$-point Clenshaw-Curtis quadrature, Fej\'{e}r's first and second rules are not difficult to be obtained.}
%The most important quadrature formulae to which Theorem 2.3 can be
%applied: the Gauss formula $I^G_n$, the left Radau formula $I^{Ra,-1}_n$, the Kronrod formula $I^{Kr}_{2n+1}$ etc. .

%%%%%%%%%%%%%%%%%%%%%%%%%%%%%%%%%%%%%%%%%%%%%%%%%%%%%%%%%%%%%%%%%%%%%%%%%%%%%%%%%%%%%%%%%%%%%%%%%%%

\section{Clenshaw-Curtis and Fej\'{e}r quadrature involving Jacobi weights}
 Fej\'{e}r \cite{Fejer} in 1933 suggested using the zeros of a Chebyshev polynomial of first or second
kind as interpolation points for quadrature rules of the form (1.2).  Here we consider the generalized  Fej\'{e}r and Clenshaw-Curtis quadrature. \textbf{Fej\'{e}r's first rule} uses the zeros of the Chebyshev polynomial $T_n(x)$ of the first kind (also called classic Chebyshev points \cite{Davis,Sloan2,Sloan3})
$$
 I_n^{F_1}[f]=\int_{-1}^1w(x)q_{n-1}^1(x)dx=\sum_{j=0}^{n-1}b_j^1\int_{-1}^1w(x)T_j(x)dx,
$$
where $q_{n-1}^1$ is the interpolation polynomial defined by
$$
\quad q_{n-1}^1(x)=\sum_{j=0}^{n-1}b_j^1T_j(x), \quad q_{n-1}^1(y_j)=f(y_j),\,\,\, y_j=\cos\left(\frac{(2j-1)\pi}{2n}\right),\, j=1,2,\ldots,n,
$$
while \textbf{Fej\'{e}r's second rule} uses the zeros of the Chebyshev polynomial $U_{n}(x)$ of the second
kind (also called Filippi points \cite{Davis})
$$
I_n^{F_2}[f]=\int_{-1}^1w(x)q_{n-1}^2(x)dx=\sum_{j=0}^{n-1}b_j^2\int_{-1}^1w(x)T_j(x)dx,
$$
where $q_{n-1}^2(x)$ is defined by
$$
\quad q_{n-1}^2(x)=\sum_{j=0}^{n-1}b_j^2T_j(x), \quad q_{n-1}^2(x_j)=f(x_j),\quad x_j=\cos\left(\frac{j\pi}{n+1}\right),\quad j=1,2,\ldots,n.
$$
\textbf{Clenshaw-Curtis quadrature} (c.f. Clenshaw-Curtis \cite{Clenshaw}) is to use the above Chebyshev points with  $n-1$ instead of $n+1$ including the endpoints $-1$ and $1$\footnote{This set of points are also called Clenshaw-Curtis points, Chebyshev extreme points or practical Chebyshev points \cite{Davis,Trefethen1,Sloan2,Sloan3}.}:
$$
I_n^{C\texttt{-}C}[f]=\int_{-1}^1w(x)q_{n-1}^3(x)dx=\sum_{j=0}^{n-1}b_j^3\int_{-1}^1w(x)T_j(x)dx,
$$
where $q_{n-1}^3$ is defined by
$$
\quad q_{n-1}^3(x)=\sum_{j=0}^{n-1}b_j^3T_j(x), \quad q_{n-1}^3(\overline{x}_j)=f(\overline{x}_j),\quad \overline{x}_j=\cos\left(\frac{j\pi}{n-1}\right),\,\, j=0,1,\ldots,n-1.
$$
The coefficients $b_j^i$ ($i=1,2,3$) in the above three interpolation polynomials can be fast computed  by FFT (c.f. Dahlquist and Bj\"{o}rck \cite{dahlq},  Trefethen \cite{Trefethen1}, Waldvogel \cite{Waldvogel} and Xiang et al. \cite{Xiang,Xiang3}).

 In addition, the modified moments $\int_{-1}^1w(x)T_j(x)dx$ can be efficiently evaluated by recurrence formulae  for Jacobi weights or Jacobi weights multiplied by $\ln((x+1)/2)$ (c.f. Piessens and Branders \cite{Piessens}).

\begin{itemize}
\item $w(x)=(1-x)^{\alpha}(1+x)^{\beta}$:
The recurrence formulae for the evaluation of the modified moments
\begin{equation}\label{moment1}
  M_k(\alpha,\beta)=\int_{-1}^1w(x)T_k(x)dx,\quad w(x)=(1-x)^{\alpha}(1+x)^{\beta}
\end{equation}
are
\begin{equation}\label{compmoment1}\quad\,\, \mbox{\footnotesize
  $(\beta+\alpha+k+2)M_{k+1}(\alpha,\beta)+2(\alpha-\beta)M_k(\alpha,\beta)
  + (\beta+\alpha-k+2)M_{k-1}(\alpha,\beta)=0$}
\end{equation}
with
{\footnotesize$$
M_0(\alpha,\beta)=2^{\beta+\alpha+1}\frac{\Gamma(\alpha+1)\Gamma(\beta+1)}{\Gamma(\beta+\alpha+2)},\quad M_1(\alpha,\beta)=2^{\beta+\alpha+1}\frac{\Gamma(\alpha+1)\Gamma(\beta+1)}{\Gamma(\beta+\alpha+2)}\frac{\beta-\alpha}{\beta+\alpha+2}.
$$}
\noindent Furthermore, the asymptotic expression is given by using the asymptotic theory of Fourier coefficients (c.f. Lighthill \cite{Lighthill}) as
$$\label{asymoment1}{\small \begin{array}{lll}
  &&M_k(\alpha,\beta)\\&\sim&-2^{\beta-\alpha}\cos(\pi\alpha)\Gamma(2\alpha+2)[k^{-2-2\alpha}+O(k^{-2\alpha-4})]\\
  &&+  (-1)^{k+1}2^{\alpha-\beta}\cos(\pi\beta)\Gamma(2\alpha+2)[k^{-2-2\beta}+O(k^{-2\beta-4})],\quad k\rightarrow \infty.\end{array}}
$$
The forward recursion is perfectly numerically stable, except in two cases:
\begin{eqnarray}
            \alpha>\beta\quad \mbox{and\quad}\beta&=&-\frac{1}{2},\frac{1}{2},\frac{3}{2},\ldots, \\
                          \beta>\alpha\quad \mbox{and\quad}\alpha&=&-\frac{1}{2},\frac{1}{2},\frac{3}{2},\ldots.
                             \end{eqnarray}

\item $w(x)=\ln((x+1)/2)(1-x)^{\alpha}(1+x)^{\beta}$:
For
\begin{equation}\label{moment2}
  G_k(\alpha,\beta)=\int_{-1}^1\ln((x+1)/2)(1-x)^{\alpha}(1+x)^{\beta}T_k(x)dx,
\end{equation}
the recurrence formulae are
\begin{equation}\label{compmoment2}{\small\begin{array}{lll}
  &&(\beta+\alpha+k+2)G_{k+1}(\alpha,\beta)+2(\alpha-\beta)G_k(\alpha,\beta)\\
  &&\quad + (\beta+\alpha-k+2)G_{k-1}(\alpha,\beta)=2M_{k}(\alpha,\beta)-M_{k-1}(\alpha,\beta)-M_{k+1}(\alpha,\beta)\end{array}}
\end{equation}
with
{\small$$
G_0(\alpha,\beta)=-2^{\beta+\alpha+1}\Phi(\alpha,\beta+1),\quad G_1(\alpha,\beta)=-2^{\beta+\alpha+1}[2\Phi(\alpha,\beta+2)-\Phi(\alpha,\beta+1)],
$$}
where
$$
\Phi(\alpha,\beta)=B(\alpha+1,\beta)[\Psi(\alpha+\beta+1)-\Psi(\beta)],
$$
$B(x,y)$ is the Beta function and $\Psi(x)$ is the Psi function (c.f. Abramowitz and Stegun \cite{Abram}). Additionally, the asymptotic expression is given by using the asymptotic theory of Fourier coefficients   as
$$\label{asymoment2}\begin{array}{lll}
 && G_k(\alpha,\beta)\\
  &\sim &(-1)^{k+1}2^{\alpha-\beta+1}\cos(\pi\beta)\Gamma(2\beta+2)k^{-2-2\beta}[-\ln 2k+\Psi(2\beta+2)\\
  &&-\frac{\pi}{2}\tan\pi\beta]-2^{\beta-\alpha-2}\cos(\pi\alpha)\Gamma(2\alpha+4)k^{-4-2\alpha},\quad k\rightarrow \infty.\end{array}
$$
The forward recursion is also perfectly numerically stable the same as that for (3.2). For more details, see Piessens and Branders \cite{Piessens}.
\end{itemize}

\vspace{0.36cm}
The convergence for the generalized $n$-point Clenshaw-Curtis quadrature, Fej\'{e}r's first and second rules, for
$$
I[f]=\int_{-1}^1k(x)f(x)dx
$$
with $\int_{-1}^1|k(x)|^pdx<\infty$ for some $p>1$, has been extensively studied in Elliott and Paget \cite{Elliott}, Sloan \cite{Sloan} and Sloan and Smith \cite{Sloan2,Sloan3}, etc. Taking into the Banach-Steinhaus (or uniform boundedness) theorem, using the convergence of Fourier series and Marcinkiewicz's inequality \cite[Vol. 2, pp. 28-30]{Marcinkiewicz}, Sloan \cite{Sloan} and Sloan and Smith \cite{Sloan2} showed that the
sums of the absolute values of the weights in (1.2) for the $n$-point Clenshaw-Curtis and Fej\'{e}r's first rule are uniformly bounded, i.e.
\begin{equation}
\lim_{n\rightarrow \infty}\sum_{j=1}^n|w_j|=\int_{-1}^1|k(x)|dx,
\end{equation}
and extended to the point set $\left\{\cos\left(\frac{2(i-1) \pi}{2n-1}\right)\right\}_{i=1}^{n}$. Identity (3.7) is also satisfied by $I_n^{F_2}[f]$.
% These results show that $I_n^{F_1}[T_k]$ and $I_n^{C\texttt{-}C}[T_k]$ are uniformly bounded independent of $n$  for $k=1,2,\ldots$.

\begin{lemma}
Suppose
$I[f]=\int_{-1}^1k(x)f(x)dx$ with $\int_{-1}^1|k(x)|^pdx<\infty$ for some $p>1$,
then the weights of $I_n^{F_2}[f]$ satisfy (3.7).
\end{lemma}
\begin{proof}
Since the weights of $I_n^{F_2}[f]$ can be represented as
$$
w_i=\frac{2\sin\theta_i}{n+1}\sum_{j=0}^{n-1}\sin((j+1)\theta_i)\overline{b}_{j+1},\quad \overline{b}_{i}=\int_{-1}^1k(x)U_{i-1}(x)dx,\quad i=1,2,\ldots,n
$$
(c.f. \cite{Sommariva}). Define an odd, $2\pi$-periodic function $K$ by
$$K(\theta)=\left\{\begin{array}{ll}\frac{\pi}{2}k(\cos \theta),&  0\le \theta\le \pi\\
-\frac{\pi}{2}k(\cos \theta),&  -\pi\le \theta <0\end{array}.\right.
$$
Then $\overline{b}_j$ has the form of
$$
\overline{b}_j=\frac{2}{\pi}\int_{0}^{\pi}K(\theta)\sin(j\theta)d\theta,\quad j=1,2,\ldots,n,
$$
which is the $j$th Fourier sine coefficient of $K$. In particular, the weight $w_i$ can be written as
$$
w_i=\frac{2\sin\theta_i}{n+1}\sum_{j=1}^{n}\sin(j\theta_i)\overline{b}_{j}=\frac{2}{n+1}S_n(\theta_i)\sin\theta_i,\quad i=1,2,\ldots,n,
$$
where $S_n(\theta)$ is the $n$th partial sum of the Fourier series for the function $K(\theta)$. Therefore, the sum of the absolute values of the weights becomes
$$
\sum_{j=1}^{n}|w_i|=\frac{2}{n+1}\sum_{j=1}^{n}|S_n(\theta_i)|\sin\theta_i,
$$
which, by directly following a similar proof to \cite{Sloan2}, establishes
$$
\lim_{n\rightarrow \infty}\sum_{j=1}^n|w_j|=\frac{2}{\pi}\int_{0}^{\pi}|K(\theta)|\sin \theta d\theta=\int_{-1}^1|k(x)|dx.
$$
\end{proof}

\begin{theorem}
Suppose $I[f]=\int_{-1}^1k(x)f(x)dx$ with $\int_{-1}^1|k(x)|^pdx<\infty$ for some $p>1$,
then  $\lim_{n\rightarrow \infty}I_n^{F_2}[f]=I[f]$  for all continuous functions in $[-1,1]$.
\end{theorem}
\begin{proof}
By Lemma 3.1, it directly  follows from \cite{Sloan2}.
 \end{proof}

\vspace{0.36cm}

Based upon these uniform boundedness, we see that for $f\in X^s$ with $s>0$,
$$
|E_n[f]|=|I[f]-I_n[f]|=\Big|\sum_{j=n}^{\infty}a_jE_n[T_j]\Big|\le \sum_{j=n}^{\infty}|a_j||E_n[T_j]|
$$
since $a_j=O(j^{-1-s})$ and $E_n[T_j]$ are uniformly bounded for $j\ge 0$, where $E_n$ denotes the error of the above these three $n$-point quadrature rules corresponding to the two Jacobi weight functions. Furthermore,  any rearrangement of the infinite sum $\sum_{j=n}^{\infty}|a_j||E_n[T_j]|$ converges to the same sum.

\vspace{0.36cm}

In the following, we will consider aliasing errors on the integration of the Chebyshev polynomials by these three quadrature rules, and derive the optimal general rate of convergence.

\vspace{0.26cm}
The computation of the aliasings by the Clenshaw-Curtis, Fej\'{e}r
first and second rules is much simple, which can be exactly computed from Fox and
Parker \cite[p. 67]{Fox}
{\small\begin{eqnarray}
\quad \quad & T_{2pn\pm j}(y_i)=\cos\left((2pn\pm
j)\frac{(2i+1)\pi}{2n}\right)=(-1)^p\cos\left(\frac{j(2i+1)\pi}{2n}\right)=(-1)^pT_j(y_i)
\\
& T_{2p(n+1)\pm j}(x_{\ell})=\cos\left((2p(n+1)\pm
j)\frac{\ell\pi}{n+1}\right)=\cos\left(\frac{j\ell\pi}{n+1}\right)=T_j(x_{\ell})
\end{eqnarray}}
for $i,j=0,1,\ldots,n-1$, $\ell=1,2,\ldots,n$  and $p=1,2,\ldots$ as
\begin{eqnarray}
 I_n^{F_1}[T_{2pn\pm j}]=(-1)^pI[T_{j}],\quad I_n^{F_1}[T_{2p(n-1)}]=0,\quad I_n^{F_2}[T_{2p(n+1)\pm j}]=I[T_{j}],
\end{eqnarray}
\begin{eqnarray}\quad I_n^{F_2}[T_{(2p+1)(n+1)\pm 1}]=I_n^{F_2}[T_{n}], \quad I_n^{F_2}[T_{(2p+1)(n+1)}]=I_n^{F_2}[T_{n+1}],
\end{eqnarray}
and
\begin{eqnarray}
I_n^{C\texttt{-}C}[T_{2p(n-1)\pm j}]&=I[T_{j}].
\end{eqnarray}

\begin{lemma}(Second mean value theorem for integration  \cite{Malik})
(i) If $G : [a, b] \rightarrow R$ is a positive monotonically decreasing function and $\phi : [a, b] \rightarrow R$ is an integrable function, then there exists a number $\zeta \in [a, b] $ such that
$$
\int_a^bG(x)\phi(x)dx=G(a+0)\int_a^{\zeta}\phi(x)dx, \quad G(a+0)=\lim_{x\rightarrow a^+}G(x).
$$

(ii) If $G : [a, b]\rightarrow R$ is a positive monotonically increasing function and $\phi : [a, b] \rightarrow R$ is an integrable function, then there exists a number $\zeta \in [a, b] $ such that
$$
\int_a^bG(x)\phi(x)dx=G(b-0)\int_{\zeta}^b\phi(x)dx, \quad G(b-0)=\lim_{x\rightarrow b^{-}}G(x).
$$

(iii) If $G : [a, b] \rightarrow R$ is a monotonic function and $\phi : [a, b] \rightarrow R$ is an integrable function, then there exists a number $\zeta \in [a, b] $ such that
$$
\int_a^bG(x)\phi(x)dx=G(a+0)\int_a^{\zeta}\phi(x)dx+G(b-0)\int_{\zeta}^b\phi(x)dx.
$$
\end{lemma}

\begin{lemma}
\begin{itemize}
\item $w(x)=(1-x)^{\alpha}(1+x)^{\beta}$: The modified moment satisfies
\begin{equation}
M_m(\alpha,\beta)=O\left(\frac{1}{m^{2+2\overline{\min}(\alpha,\beta)}}\right),\quad m=1,2,\ldots.
\end{equation}
Moreover, the aliasing errors by the three quadrature rules for $m=2pn\pm j$, $2p(n+1)\pm j$ or $2p(n-1)\pm j$ with respect to $I_n^{F_1}[T_m]$, $I_n^{F_2}[T_m]$ and $I_n^{C\texttt{-}C}[T_m]$ for $p=1,2,\ldots$ and $j=0,1,2,\ldots,n-1$, respectively, satisfy
\begin{equation}
|E_n[T_m]|=|M_j(\alpha,\beta)|+O\left(\frac{1}{m^{2+2\overline{\min}(\alpha,\beta)}}\right),
\end{equation}
where $\overline{\min}(\alpha,\beta)$ is defined by
$$
\overline{\min}(\alpha,\beta)=\left\{\begin{array}{ll}
0&\mbox{if $\alpha=\beta=-\frac{1}{2}$}\\
\beta&\mbox{if $\alpha=-\frac{1}{2}$ and $\beta\not=-\frac{1}{2}$}\\
\alpha&\mbox{if $\alpha\not=-\frac{1}{2}$ and $\beta=-\frac{1}{2}$}\\
\min(\alpha,\beta)&\mbox{otherwise}\end{array}.\right.
$$

\item $w(x)=\ln((x+1)/2)(1-x)^{\alpha}(1+x)^{\beta}$: For $\beta\not=-\frac{1}{2}$, the modified moment satisfies
 \begin{equation}\quad G_{m}(\alpha,\beta)=O\left(\frac{\ln 2m}{m^{2+2\beta}}\right)+O\left(\frac{1}{m^{4+2\alpha}}\right),\quad m=1,2,\ldots.
\end{equation}
The aliasing errors by the three quadrature rules for $m=2pn\pm j$, $2p(n+1)\pm j$ or $2p(n-1)\pm j$ with respect to $I_n^{F_1}[T_m]$, $I_n^{F_2}[T_m]$ and $I_n^{C\texttt{-}C}[T_m]$ for $p=1,2,\ldots$ and $j=0,1,2,\ldots,n-1$, respectively, satisfy
\begin{equation}
|E_n[T_m]|=|G_j(\alpha,\beta)|+O\left(\frac{\ln 2m}{m^{2+2\beta}}\right)+O\left(\frac{1}{m^{4+2\alpha}}\right).
\end{equation}
Particularly, for $\beta=-\frac{1}{2}$, the term  $\frac{\ln 2m}{m^{2+2\beta}}$ in (3.15) and (3.16) is replaced by $\frac{1}{m^{2+2\beta}}$, respectively.
\end{itemize}
\end{lemma}

\begin{proof}
\begin{itemize}
\item $w(x)=(1-x)^{\alpha}(1+x)^{\beta}$: By setting $x=\cos\theta$, it follows
$$\begin{array}{lll}
M_m(\alpha,\beta)&=&\int_{-1}^1(1-x)^{\alpha}(1+x)^{\beta}T_m(x)dx\\
&=& 2^{\alpha+\beta+1}\int_0^{\pi}\sin^{1+2\alpha}\frac{\theta}{2}\cos^{1+2\beta}\frac{\theta}{2}\cos m\theta d\theta\\
&=& 2^{\alpha+\beta+2}\int_0^{\frac{\pi}{2}}\sin^{1+2\alpha}t\cos^{1+2\beta}t\cos 2mt dt.\end{array}
$$

\textbf{In the case $-1<\min(1+2\alpha, 1+2\beta)<0$}: Notice that
$$
\int_0^{\frac{\pi}{2}}\sin^{1+2\alpha}t\cos^{1+2\beta}t\cos 2mt dt=\left(\int_0^{\frac{\pi}{4}}+\int_{\frac{\pi}{4}}^{\frac{\pi}{2}}\right)\sin^{1+2\alpha}t\cos^{1+2\beta}t\cos 2mt dt,
$$
where the first term on the right hand side can be estimated by
\begin{equation}\begin{array}{lll}
&&\Big|\int_0^{\frac{\pi}{4}}\sin^{1+2\alpha}t\cos^{1+2\beta}t\cos 2mt dt\Big|\\
&\le& \int_0^{\frac{1}{2m}}\sin^{1+2\alpha}t\cos^{1+2\beta}tdt+|\int_{\frac{1}{2m}}^{\frac{\pi}{4}}\sin^{1+2\alpha}t\cos^{1+2\beta}t\cos 2mt dt|.\end{array}
\end{equation}
Then, by
$$0<\sin^{1+2\alpha}t\le  \max\left(1,(2/\pi)^{1+2\alpha}\right) t^{1+2\alpha},\quad 0<\cos^{1+2\beta}t\le  \max\left(1,\cos^{1+2\beta}1\right)$$
for $t\in (0,\frac{1}{2m}]$, the first term in (3.17) can be estimated as
\begin{equation}
\int_0^{\frac{1}{2m}}\sin^{1+2\alpha}t\cos^{1+2\beta}tdt=O\left(\int_0^{\frac{1}{2m}}t^{1+2\alpha}dt\right)=O(m^{-2-2\alpha}).
\end{equation}
Moreover, the second term in (3.17) $\int_{\frac{1}{2m}}^{\frac{\pi}{4}}\sin^{1+2\alpha}t\cos^{1+2\beta}t\cos 2mt dt$ can be estimated by (iii) of Lemma 3.3 as follows
$${\small\begin{array}{lll}
&&\int_{\frac{1}{2m}}^{\frac{\pi}{4}}\sin^{1+2\alpha}t\cos^{1+2\beta}t\cos 2mt dt\\
&=&\sin^{1+2\alpha}\frac{1}{2m}\int_{\frac{1}{2m}}^{\zeta}\cos^{1+2\beta}t\cos 2mt dt+\sin^{1+2\alpha}{\frac{\pi}{4}}\int_{\zeta}^{\frac{\pi}{4}}\cos^{1+2\beta}t\cos 2mtdt\\
%&=&O(m^{-2-2\alpha})+O(m^{-1})
\end{array}}
$$
for some $\zeta\in [\frac{1}{2m},\frac{\pi}{4}]$ by using the monotonicity of $\sin^{1+2\alpha}t$ on $[\frac{1}{2m},\frac{\pi}{4}]$. Applying (iii) of Lemma 3.3 again for $\cos^{1+2\beta}t$ on $[\frac{1}{2m},\frac{\pi}{4}]$, we obtain
$$\int_{\frac{1}{2m}}^{\zeta}\cos^{1+2\beta}t\cos 2mtdt=O(m^{-1}),\quad \int_{\zeta}^{\frac{\pi}{4}}\cos^{1+2\beta}t\cos 2mtdt=O(m^{-1})$$
and then we get $$\int_{\frac{1}{2m}}^{\frac{\pi}{4}}\sin^{1+2\alpha}t\cos^{1+2\beta}t\cos 2mt dt=O(m^{-2-2\alpha})+O(m^{-1}),$$
which, combining (3,17) and (3.18), yields
$$\int_0^{\frac{\pi}{4}}\sin^{1+2\alpha}t\cos^{1+2\beta}t\cos 2mt dt=O(m^{-2-2\alpha})+O(m^{-1}).$$

Similarly, by setting $u=\frac{\pi}{2}-t$, it yields $\int_{\frac{\pi}{4}}^{\frac{\pi}{2}}\sin^{1+2\alpha}t\cos^{1+2\beta}t\cos 2mt dt=O(m^{-2-2\beta})+O(m^{-1})$ and then
$M_m(\alpha,\beta)=O(m^{-2-2\min(\alpha,\beta)})$.

\vspace{0.26cm}
\textbf{In the case $\min(1+2\alpha, 1+2\beta)=0$}: Without loss of generality, assume $\beta=-\frac{1}{2}$. Then
$M_m(\alpha,\beta)=\int_0^{\frac{\pi}{2}}\sin^{1+2\alpha}t\cos 2mt dt$. The special case for $\alpha=\beta=-\frac{1}{2}$ follows directly from $M_m(-\frac{1}{2},-\frac{1}{2})=\int_{0}^{\pi}\cos mtdt=0$. In the other case, $M_m(\alpha,\beta)$ can be reduced to the case $-1<\min(1+2\alpha, 1+2\beta)<0$ by integrating by parts at most $\lfloor 1+2\alpha\rfloor$ times, and then (3.13) follows from a similar way for this case.

\vspace{0.26cm}
 Similarly,\textbf{ in the case $0<\min(1+2\alpha, 1+2\beta)\le 1$}, integrating by parts once follows the desired result. Thus, by induction we get (3.13) for $\min(1+2\alpha, 1+2\beta)>-1$.

\vspace{0.26cm}
Expression (3.14) directly follows from the aliasings (3.10-3.12) and the asymptotics on $M_m(\alpha,\beta)$.

\vspace{0.36cm}
\item $w(x)=\ln((x+1)/2)(1-x)^{\alpha}(1+x)^{\beta}$: Similarly, by setting $x=\cos\theta$ it follows
\begin{equation}\begin{array}{lll}
&&G_m(\alpha,\beta)\\&=& 2^{\alpha+\beta+3}\int_0^{\frac{\pi}{2}}\ln (\cos t)\sin^{1+2\alpha}t\cos^{1+2\beta}t\cos 2mt dt\\
&=& 2^{\alpha+\beta+3}\left(\int_0^{\frac{1}{2m}}+\int_{\frac{1}{2m}}^{\frac{\pi}{2}-{\frac{1}{2m}}}+\int_{\frac{\pi}{2}-{\frac{1}{2m}}}^{\frac{\pi}{2}}\right)\ln (\cos t)\sin^{1+2\alpha}t\cos^{1+2\beta}t\cos 2mt dt.\end{array}
\end{equation}

\vspace{0.26cm}
\textbf{In the case $-1<1+2\beta<0$}: By using $\ln (\cos t)=\ln (1-2\sin^2\frac{t}{2})=O(t^2)$, $\sin^{1+2\alpha}t=O(t^{1+2\alpha})$ and $\sin^{1+2\beta}t=O(t^{1+2\beta})$ for $t\in (0,\frac{\pi}{2}]$, we have the following estimates on the first and third terms in (3.19), respectively,
$$
\int_0^{\frac{1}{2m}}\ln (\cos t)\sin^{1+2\alpha}t\cos^{1+2\beta}t\cos 2mt dt=O\left(\int_0^{\frac{1}{2m}}t^{3+2\alpha}dt\right)=O(m^{-4-2\alpha})
$$
and
$$\begin{array}{lll}&&\int_{\frac{\pi}{2}-{\frac{1}{2m}}}^{\frac{\pi}{2}}\ln (\cos t)\sin^{1+2\alpha}t\cos^{1+2\beta}t\cos 2mt dt\\
&=&(-1)^m\int_0^{\frac{1}{2m}}\ln (\sin t)\cos^{1+2\alpha}t\sin^{1+2\beta}t\cos 2mt dt\\
&=&O\left(\Big|\int_0^{\frac{1}{2m}}u^{1+2\beta}\ln u du\Big|\right)\\
&=&O(m^{-2-2\beta}\ln 2m).\end{array}
$$
While for the second term in (3.19), integrating by parts we get
$$\begin{array}{lll}&&\int_{\frac{1}{2m}}^{\frac{\pi}{2}-{\frac{1}{2m}}}\ln (\cos t)\sin^{1+2\alpha}t\cos^{1+2\beta}t\cos 2mt dt\\
&=&\frac{1}{2m}\ln (\cos t)\sin^{1+2\alpha}t\cos^{1+2\beta}t\sin 2mt\Big |_{\frac{1}{2m}}^{\frac{\pi}{2}-{\frac{1}{2m}}}\\
&&-\frac{1}{2m}\int_{\frac{1}{2m}}^{\frac{\pi}{2}-{\frac{1}{2m}}}\left\{-\sin^{2+2\alpha}t\cos^{2\beta}t(1+(1+2\beta)\ln \cos t)\right.\\
&&\left.\quad\quad\quad\quad\quad\quad+(1+2\alpha)\ln (\cos t)\sin^{2\alpha}t\cos^{2+2\beta}t\right\}\sin 2mt dt\\
&=&O(m^{-2-2\beta}\ln 2m)+O(m^{-4-2\alpha})+Z_1-Z_2,\end{array}
$$
where
$$
Z_1=\frac{1}{2m}\int_{\frac{1}{2m}}^{\frac{\pi}{2}-{\frac{1}{2m}}}\sin^{2+2\alpha}t\cos^{2\beta}t(1+(1+2\beta)\ln \cos t)\sin 2mt dt
$$
can be estimated by (ii) of Lemma 3.3 for some $\eta\in \left[\frac{1}{2m},\frac{\pi}{2}-{\frac{1}{2m}}\right]$,
$$\sin^{2+2\alpha}t=O(1), \,\, \cos^{2\beta}t=O(m^{-2\beta}),\,\, 1+(1+2\beta)\ln \cos t=O(\ln 2m)$$
for $t\in \left[\frac{1}{2m},\frac{\pi}{2}-{\frac{1}{2m}}\right]$,
and  $$\int_a^b \sin 2mt dt=O(m^{-1}),\quad \forall a,b\in \left[0,\frac{\pi}{2}\right],$$ as
$$\begin{array}{lll}
Z_1&=&\frac{1}{2m}\sin^{2+2\alpha}t\cos^{2\beta}t(1+(1+2\beta)\ln \cos t)\Big|_{t=\frac{\pi}{2}-{\frac{1}{2m}}}\int_{\eta}^{\frac{\pi}{2}-{\frac{1}{2m}}}\sin 2mt dt\\
&=&O(m^{-2-2\beta}\ln 2m).\end{array}
$$
Similarly, we obtain
$$\begin{array}{lll}
Z_2:&=&-\frac{1}{2m}\int_{\frac{1}{2m}}^{\frac{\pi}{2}-{\frac{1}{2m}}}(1+2\alpha)\ln (\cos t)\sin^{2\alpha}t\cos^{2+2\beta}t\sin 2mt dt\\
&=&O\left(m^{-1}\int_{\frac{1}{2m}}^{\frac{\pi}{2}-{\frac{1}{2m}}}t^{2+2\alpha}\cos^{2+2\beta}t\sin 2mt dt\right)\\
&=&O(m^{-2}),\end{array}
$$
which together indicates $G_m(\alpha,\beta)=O(m^{-2-2\beta}\ln 2m)+O(m^{-4-2\alpha})$.

\vspace{0.26cm}
Particularly, in the case $\beta=-\frac{1}{2}$, $G_m(\alpha,-\frac{1}{2})$ can be estimated by (3.19) with $\frac{\pi}{4}$ instead of $\frac{1}{2m}$ as
\begin{equation}
\begin{array}{lll}
&&G_m(\alpha,-\frac{1}{2})\\&=&O(m^{-4-2\alpha})+2^{\alpha+5/2}\int_{\frac{\pi}{4}}^{\frac{\pi}{2}}\ln (\cos t)\sin^{1+2\alpha}t\cos 2mt dt\\
&=&O(m^{-4-2\alpha})+(-1)^m2^{\alpha+5/2}\int_0^{\frac{\pi}{4}}\ln (\sin t)\cos^{1+2\alpha}t\cos 2mt dt\\
&=&O(m^{-4-2\alpha})+O\left(\int_0^{\frac{\pi}{4}}\ln (t)\cos^{1+2\alpha}t\cos 2mt dt\right),
\end{array}
\end{equation}
where the second term in (3.20) can be estimated by
$$\begin{array}{lll}
&&\int_0^{\frac{\pi}{4}}\ln (t)\cos^{1+2\alpha}t\cos 2mt dt\\
&=&\frac{1}{2m}\ln(\frac{\pi}{4})\cos^{1+2\alpha}\frac{\pi}{4}\sin \frac{m\pi}{2}-\frac{1}{2m}\int_0^{\frac{\pi}{4}}\frac{\cos^{1+2\alpha}t}{t}\sin 2mt dt\\
&&+\frac{1+2\alpha}{2m}\int_0^{\frac{\pi}{4}}\ln(t)\sin t\cos^{2\alpha}t\sin 2mt dt\\
&=&O(m^{-1})-\frac{1}{2m}\int_0^{\frac{m\pi}{2}}\cos^{1+2\alpha}(\frac{u}{2m})\frac{\sin u}{u}du+O\left(m^{-1}\int_0^{\frac{\pi}{4}}|t\ln(t)|dt\right)\\
&=&-\frac{1}{2m}\left(\int_0^{\xi}\frac{\sin u}{u}du+\cos^{1+2\alpha}(\frac{\pi}{4})\int_{\xi}^{\frac{m\pi}{2}}\frac{\sin u}{u}du\right)+O(m^{-1})
\end{array}
$$
for some $\xi\in [0,\frac{m\pi}{2}]$, which, together with $\int_0^{+\infty}\frac{\sin u}{u}du=\frac{\pi}{2}$ and (3.20), implies
$G_m(\alpha,-\frac{1}{2})=O(m^{-4-2\alpha})+O(m^{-1})$.

\vspace{0.26cm}
For the general cases, applying  similar arguments as those for $w(x)=(1-x)^{\alpha}(1+x)^{\beta}$ gives the desired result (3.15) by induction.

\vspace{0.26cm}
Expression (3.16) directly follows from the aliasings and the asymptotics on $G_m(\alpha,\beta)$.
\end{itemize}
\end{proof}

\begin{theorem}
If $f \in X^s$ for $s>0$, the convergence of $n$-point Clenshaw-Curtis quadrature, Fej\'{e}r's first and second rules
has  the rate
\begin{itemize}
\item for $w(x)=(1-x)^{\alpha}(1+x)^{\beta}$:
\begin{equation}
E_n[f]=\left\{\begin{array}{ll}
O(n^{-s-1})&\mbox{if $\min(\alpha,\beta)\ge -\frac{1}{2}$}\\
O(n^{-s-2-2\min(\alpha,\beta)})&\mbox{if $-1<\min(\alpha,\beta)< -\frac{1}{2}$}\end{array};\right.
\end{equation}
\item for $w(x)=\ln((x+1)/2)(1-x)^{\alpha}(1+x)^{\beta}$:
\begin{equation}
E_n[f]=\left\{\begin{array}{ll}
O(n^{-s-1})&\mbox{if $\beta> -\frac{1}{2}$}\\
O(n^{-s-2-2\beta}\ln n)&\mbox{if $-1<\beta\le -\frac{1}{2}$}\end{array}.\right.
\end{equation}
\end{itemize}
\end{theorem}
\begin{proof}
Here we only prove (3.22) for $I_n^{F_1}[f]$ for $w(x)=\ln((x+1)/2)(1-x)^{\alpha}(1+x)^{\beta}$. Similar proofs can
be applied to prove (3.21) and other cases in (3.22).

With $f\in X^s$, we see that $${\displaystyle E^{F_1}_n [f]= \sum_{m=n}^{\infty}a_m E^{F_1}_n [T_m]}$$ is uniformly and absolutely convergent since $a_m = O(m^{-s-1})$ and $E^{F_1}_n [T_m]$ are uniformly bounded independent of $n$ and $m$. Moreover,  $E^{F_1}_n [f]$ can be estimated by
$$
|E^{F_1}_n [f]|\le \sum_{m=n}^{\infty}|a_m| |E^{F_1}_n [T_m]| = S_0+S_3,
$$
with
$$
S_0 ={\displaystyle\sum_{p=1}^{\infty}\sum_{0<|j|<n}|a_{2pn+j}||E_n^{F_1}[T_{\ell 2pn+j}]|},\quad\quad
S_3 ={\displaystyle\sum_{\ell=1}^{\infty}|a_{\ell n}||E_n^{F_1}[T_{\ell n}]|}.
$$
From the aliasing  (3.10), we find
$$\begin{array}{lll}
S_3 &=&{\displaystyle\sum_{\ell=1}^{\infty}|a_{\ell n}||E_n^{F_1}[T_{\ell n}]|}\\
&\le&{\displaystyle\sum_{p=1}^{\infty}\left\{|a_{2pn}| \cdot (|G_{2pn}(\alpha,\beta)|+|G_0(\alpha,\beta)|)+|a_{(2p-1)n}| \cdot (|G_{(2p-1)n}(\alpha,\beta)|+|I_n^{F_1}[T_n]|)\right\}}\\
&=&{\displaystyle\sum_{p=1}^{\infty}\frac{O(1)}{(2pn )^{s+1}}}\\
&=&{\displaystyle O(n^{-s-1})} \end{array}$$
since $G_{2pn}(\alpha,\beta)$, $G_{(2p-1)n}(\alpha,\beta)$ and $I_n^{F_1}[T_n]$ are uniformly bounded from (3.15).

Additionally, $S_0$ can be estimated by $S_0 \le S_1+S_2$  according to the aliasing errors (3.16)
$$
|E_n^{F_1}[T_m]|=|G_j(\alpha,\beta)|+O\left(\frac{\ln 2m}{m^{2+2\beta}}\right)+O\left(\frac{1}{m^{4+2\alpha}}\right)
$$
as follows with
$$\begin{array}{lll}
S_1 &=&{\displaystyle\sum_{p=1}^{\infty}\sum_{0<|j|<n}|a_{2pn+j}| \cdot |G_{|j|}(\alpha,\beta)|}\\
&=&{\displaystyle \sum_{p=1}^{\infty}\sum_{0<|j|<n}O\left(\frac{1}{(2pn +j)^{s+1}}\right)|G_{|j|}(\alpha,\beta)|}\\
&=&{\displaystyle\sum_{p=1}^{\infty}\frac{1}{p^{1+s}}}\cdot\left\{\begin{array}{ll}O(n^{-s-1})\int_1^n(x^{-2-2\beta}\ln x+x^{-4-2\alpha})dx&\mbox{if $\beta\not=-\frac{1}{2}$}\\
O(n^{-s-1})\int_1^n(x^{-1}+x^{-4-2\alpha})dx&\mbox{if $\beta=-\frac{1}{2}$}\end{array}\mbox{(by (3.15))}\right.\\
&=&\left\{\begin{array}{ll}O(n^{-s-1})&\mbox{if $\beta>-\frac{1}{2}$}\\
O(n^{-s-2-2\beta}\ln n)&\mbox{if $\beta\le-\frac{1}{2}$}\end{array}\right.\end{array}
$$
and
$$\begin{array}{lll}
S_2 &=&{\displaystyle\sum_{p=1}^{\infty}\sum_{0<|j|<n}|a_{2pn+j}|\cdot\left\{\begin{array}{ll}O\left(\frac{\ln2(2pn+j)}{(2pn+j)^{2+2\beta}}\right)
+O\left(\frac{1}{(2pn+j)^{4+2\alpha}}\right),&\mbox{if $\beta\not=-\frac{1}{2}$}\\
O\left(\frac{1}{(2pn+j)}\right)+O\left(\frac{1}{(2pn+j)^{4+2\alpha}}\right),&\mbox{if $\beta=-\frac{1}{2}$}\end{array}\right.}\\
&=&{\displaystyle\sum_{\mbox{$m\ge n$ and $m\not\in \Gamma$}}O(m^{-s-1})\cdot\left\{\begin{array}{ll}O\left(\frac{\ln(2m)}{m^{2+2\beta}}\right)+O\left(\frac{1}{m^{4+2\alpha}}\right),&\mbox{if $\beta\not=-\frac{1}{2}$}\\
O\left(\frac{1}{m}\right)+O\left(\frac{1}{m^{4+2\alpha}}\right),&\mbox{if $\beta=-\frac{1}{2}$}\end{array}\right.}\\
&=&{\displaystyle\left\{\begin{array}{ll}O\left(\frac{\ln(2n)}{n^{s+2+2\beta}}+\frac{1}{n^{s+4+2\alpha}}\right),&\mbox{if $\beta\not=-\frac{1}{2}$}\\
O\left(\frac{1}{n^{s+1}}\right),&\mbox{if $\beta=-\frac{1}{2}$}\end{array}\,(\Gamma=\{pn\,|\,p=1,2,\ldots\}),
\right.}\end{array}
$$
Combining these estimates, we obtain  (3.22) for the $n$-point Fej\'{e}r's first rule.
\end{proof}

\vspace{0.36cm}
The optimal general convergence rates of these three quadrature  rules can be verified by using $f(x)=|x-0.5|^s$ ($f\in X^s$ with $s>0$ not an even number). Figures 3.1-3.2 illustrate the convergence rates for $n$-point Clenshaw-Curtis, Fej\'{e}r's first and second rules for Jacobi weight $w(x)=(1-x)^{\alpha}(1+x)^{\beta}$ and $f(x)=|x-0.5|^{s}$ with $s=0.6$ and $s=1.6$, compared with $n^{-s-1}$ if $\min(\alpha,\beta)\ge -\frac{1}{2}$ , and $n^{-s-2-2\min(\alpha,\beta)}$ if $-1<\min(\alpha,\beta)< -\frac{1}{2}$, respectively.

Figures 3.3-3.4 show the convergence rates by these three $n$-point quadrature with the same functions for weight $w(x)=\ln((1+x)/2)(1-x)^{\alpha}(1+x)^{\beta}$, compared with $n^{-s-1}$ if $\min(\alpha,\beta)> -\frac{1}{2}$, and $n^{-s-2-2\beta}\ln(n)$ if $-1<\min(\alpha,\beta)\le -\frac{1}{2}$, respectively.

The numerical evidence shows that Clenshaw-Curtis and  Fej\'{e}r's first and second quadrature
are of approximately equal accuracy for these two weights, and the convergence rates (3.21) and (3.22) are attainable for some functions of finite regularities.

\begin{figure}[htbp]
\centerline{\includegraphics[height=8cm,width=15cm]{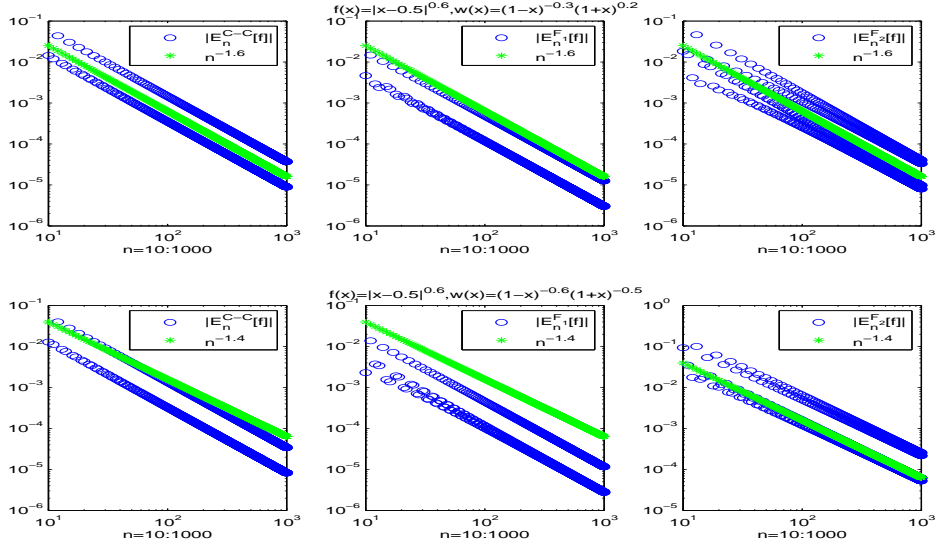}}
\caption{The absolute errors for $n$-point Clenshaw-Curtis, Fej\'{e}r's first and second rules for $f(x)=|x-0.5|^{0.6}$ ($f\in X^{0.6}$) and $w(x)=(1-x)^{\alpha}(1+x)^{\beta}$ with $\alpha=-0.3$ and $\beta=0.2$ ($1$st row), and $\alpha=-0.6$ and $\beta=-0.5$ ($2$nd row), compared with $n^{-1-0.6}$ and $n^{-0.6-2-2\min(-0.6,-0.5)}$, respectively, for $n=10:1000$.}
\end{figure}

\begin{figure}[htbp]
\centerline{\includegraphics[height=8cm,width=15cm]{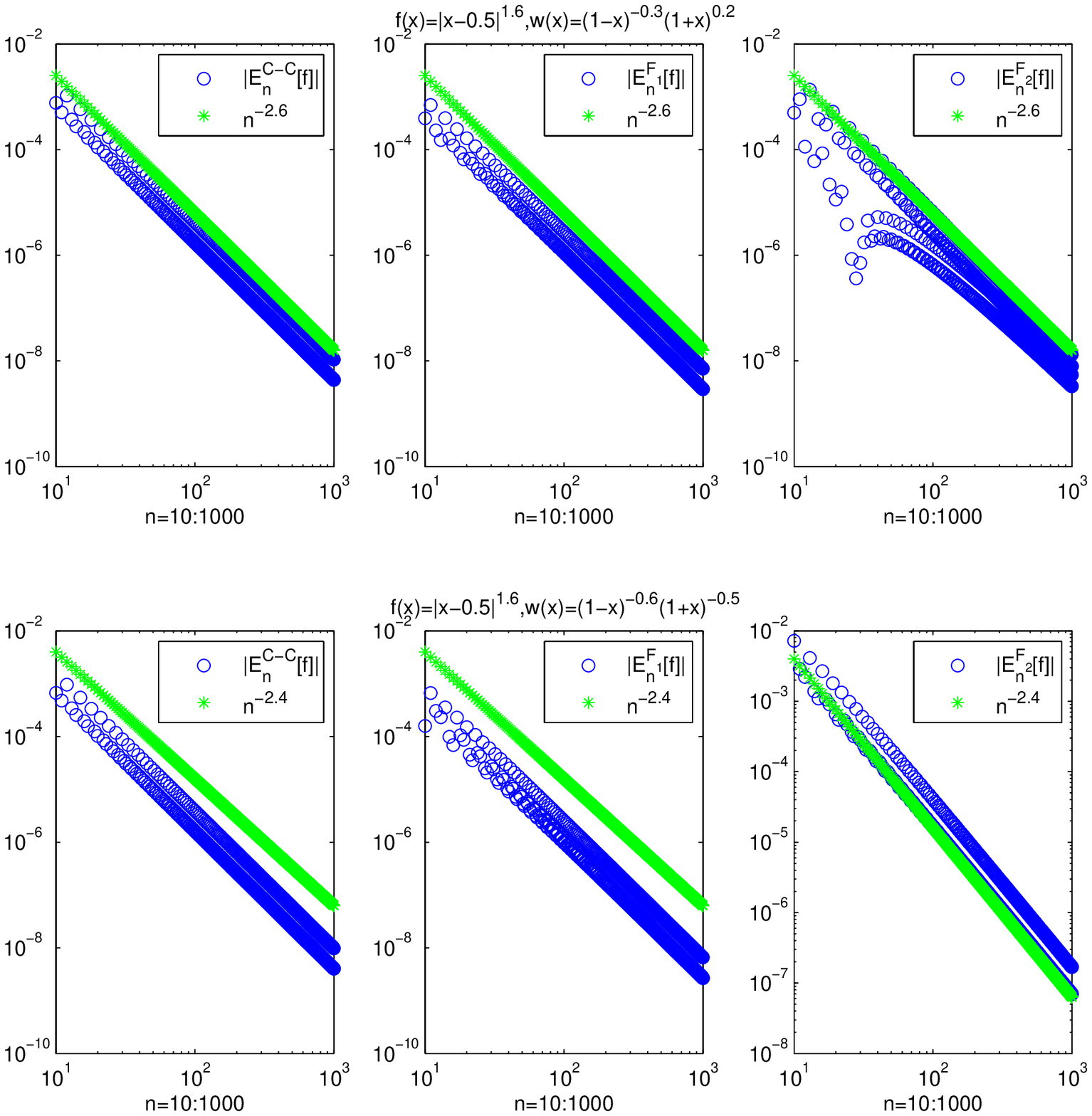}}
\caption{The absolute errors for $n$-point Clenshaw-Curtis, Fej\'{e}r's first and second rules for $f(x)=|x-0.5|^{1.6}$ ($f\in X^{1.6}$) and $w(x)=(1-x)^{\alpha}(1+x)^{\beta}$ with $\alpha=-0.3$ and $\beta=0.2$ ($1$st row), and $\alpha=-0.6$ and $\beta=-0.5$ ($2$nd row), compared with $n^{-1-1.6}$ and $n^{-1.6-2-2\min(-0.6,-0.5)}$, respectively, for $n=10:1000$.}
\end{figure}

\begin{figure}[htbp]
\centerline{\includegraphics[height=8cm,width=15cm]{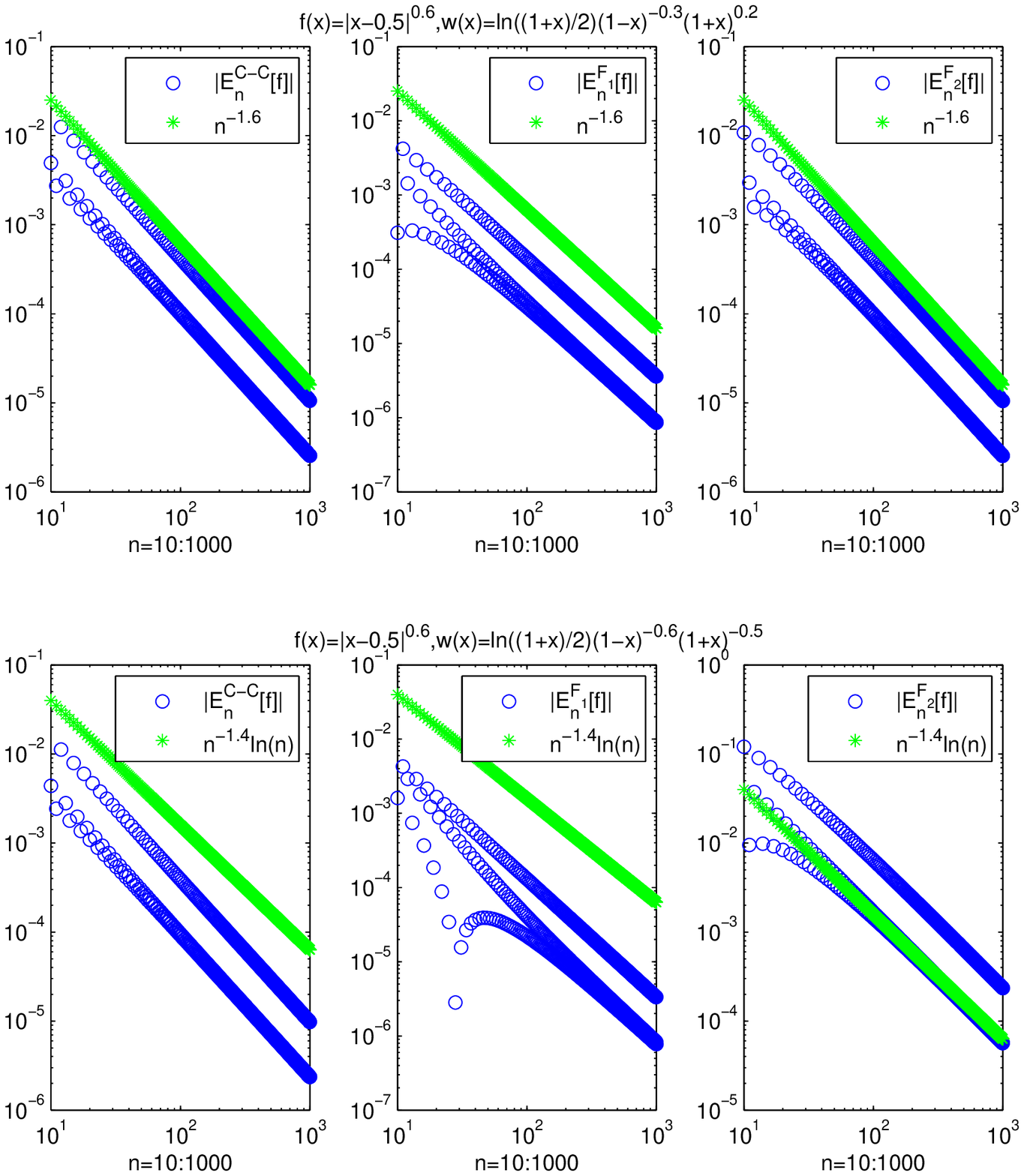}}
\caption{The absolute errors for $n$-point Clenshaw-Curtis, Fej\'{e}r's first and second rules for $f(x)=|x-0.5|^{0.6}$ ($f\in X^{0.6}$) and $w(x)=\ln((1+x)/2)(1-x)^{\alpha}(1+x)^{\beta}$ with $\alpha=-0.3$ and $\beta=0.2$ compared with $n^{-1-0.6}$ ($1$st row), and $\alpha=-0.6$ and $\beta=-0.5$ ($2$nd row), compared with $n^{-0.6-2-2\min(-0.6,-0.5)}\ln n$, respectively, for $n=10:1000$.}
\end{figure}

\begin{figure}[htbp]
\centerline{\includegraphics[height=8cm,width=15cm]{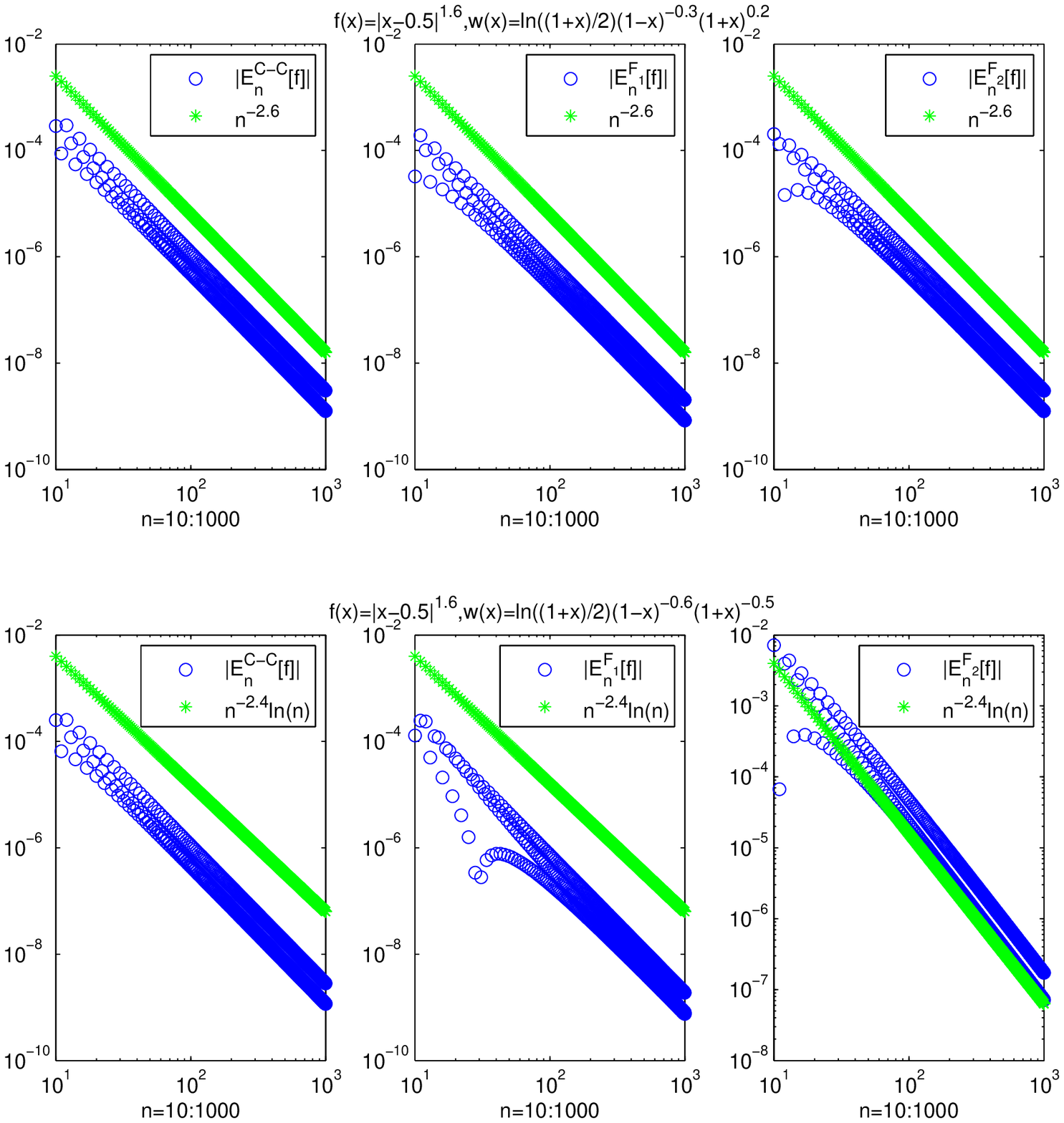}}
\caption{The absolute errors for $n$-point Clenshaw-Curtis, Fej\'{e}r's first and second rules for $f(x)=|x-0.5|^{1.6}$ ($f\in X^{1.6}$) and $w(x)=\ln((1+x)/2)(1-x)^{\alpha}(1+x)^{\beta}$ with $\alpha=-0.3$ and $\beta=0.2$ compared with $n^{-1-1.6}$ ($1$st row), and $\alpha=-0.6$ and $\beta=-0.5$ ($2$nd row), compared with $n^{-1.6-2-2\min(-0.6,-0.5)}\ln n$, respectively, for $n=10:1000$.}
\end{figure}

%%%%%%%%%%%%%%%%%%%%%%%%%%%%%%%%%%%%%%%%%%%%%%%%%%%%%%%%%%%%%%%%%%%%%%%%%%%%%%%%%%%%%%%%%%%%%%%

%%%%%%%%%%%%%%%%%%%%%%%%%%%%%%%%%%%%%%%%%%%%

\section{Final remarks}
The Peano kernel theorem provides a most useful representation of the quadrature
error for the set of bounded variation functions (c.f. Brass \cite{Brass77}, Brass and Petras \cite{BrassPetras} and Davis and Rabinowitz \cite{Davis}). Based on the Peano kernel theorem and the estimates on the kernel function (c.f. Freud \cite{Freud}), Brass and Petras \cite{BrassPetras} obtained the error bound for any quadrature with positive quadrature weights (also see Diethelm \cite{Diethelm}).

\begin{theorem} (Brass and Petras \cite{BrassPetras})
Suppose $w(x)$ is a nonnegative and integrable weight
function satisfies
\begin{equation}\sup_{-1\le x\le 1}w(x)(1-x^2)^{1/2}<\infty, \end{equation}
and $E_n[{\cal P}_{n-1}]=0$ for any  positive interpolatory quadrature
formula $I_n$ with $n$ nodes\footnote{$E_n[{\cal P}_{n-1}]=0$ means $E_n[p]=I[p]-I_n[p]=0$ for all $p\in {\cal P}_{n-1}$.},  where ${\cal P}_{n-1}$ denotes the set of polynomials with degree less than $n-1$. If $f(x)$ has an
absolutely continuous $(k-1)$st derivative $f^{(k-1)}$ on $[-1,1]$
(if $k\ge 1$) and a $k$th derivative $f^{(k)}$ of bounded variation
$V_k$, then the quadrature error satisfies
\begin{equation}\label{errorGauss}
 E_n[f]= O(n^{-k-1}).
 \end{equation}

\end{theorem}

Thus, $n$-point Gauss quadrature for the weight function satisfying (4.1) has the convergence rate (4.2). Particularly, the rate (4.2) can be achieved for functions of the form of
$$f^{(k)}(x)=\left\{\begin{array}{ll}
0&\mbox{if $-1\le x\le \eta$}\\
M&\mbox{if $\eta< x\le 1$}\end{array},\quad M\not=0,\right.
$$
where $\eta$ is  chosen so that $|K_{k+1}(\eta)|=\|K_{k+1}\|_{\infty}$ and $K_{k+1}$ is the $(k+1)$th Peano kernel function (c.f Brass and Petras \cite [p. 87]{BrassPetras}). Then for this set of functions, the rate (4.2) is optimal.

However, the optimal convergence rate could be missed for such function which is of $k$th bounded variation with $\int_{-1}^1|f^{(k+1)}(x)|dx<\infty$ but the $(k+1)$th bounded variation does not exist, for example, $f_*^{(k)}(x)=\sqrt{1-x^2}$ (${\rm Var} (f_*^{(k)})<\infty$, $f_*\in X^{k+1}$), $f_{\gamma}^{(k)}(x)=|x-c|^{\gamma}$ (${\rm Var} (f_{\gamma}^{(k)})<\infty$, $f_{\gamma}\in X^{k+\gamma}$, $-1<c< 1$, $0<\gamma<1$) and
$$f_s(x)=\left\{\begin{array}{ll}
0&\mbox{if $-1\le x\le \xi$}\\
(x-\xi)_+^s&\mbox{if $\xi< x\le 1$}\end{array},\right.\,\, -1<\xi<1,
$$
where $s>0$ is a non-integer, ${\rm Var} (f_s^{(\lfloor s \rfloor)})<\infty$ and $f_s\in X^{s}$)

 In addition, the convergence rate (\ref{errorGauss}) can not be applied to the case $${\displaystyle\sup_{-1< x< 1}w(x)(1-x^2)^{1/2}=\infty}.$$

Comparing Theorem 3.5 and Theorem 4.1, we see that the convergence orders in Theorem 3.5 on the above special functions by $n$-point Clenshaw-Curtis quadrature, Fej\'{e}r's first and second rules
can be estimated  higher than those by $n$-point Gauss quadrature given in Theorem 4.1 for $w(x)=(1-x)^{\alpha}(1+x)^{\beta}$ with $\alpha,\beta\ge -\frac{1}{2}$.

Nevertheless, numerical evidence shows that for Jacobi weight $w(x)=(1-x)^{\alpha}(1+x)^{\beta}$ ($\alpha,\beta>-1$), $n$-point Gauss quadrature enjoys the same convergence rate (3.21) as that for $n$-point Clenshaw-Curtis and  Fej\'{e}r's quadrature, and is of approximately equal accuracy. For simplicity, here we only consider comparisons between Gauss and Clenshaw-Curtis quadrature for $f(x)=|x-0.5|^{s}$ ($f\in X^{s}$, $s=0.6,1.6$) and $w(x)=(1-x)^{\alpha}(1+x)^{\beta}$ with $\alpha=-0.3$ and $\beta=0.2$,  and $\alpha=-0.6$ and $\beta=-0.5$, respectively: $n=10:1000$ (see Figure 4.1). Based on these numerical evidence, we put an open problem at the end.

{\bf Open problem}. $n$-point Gauss quadrature enjoys the same convergence rate (3.21) for Jacobi weight $w(x)=(1-x)^{\alpha}(1+x)^{\beta}$ for $f\in X^s$.

\begin{figure}
\centerline{\includegraphics[height=6cm,width=14cm]{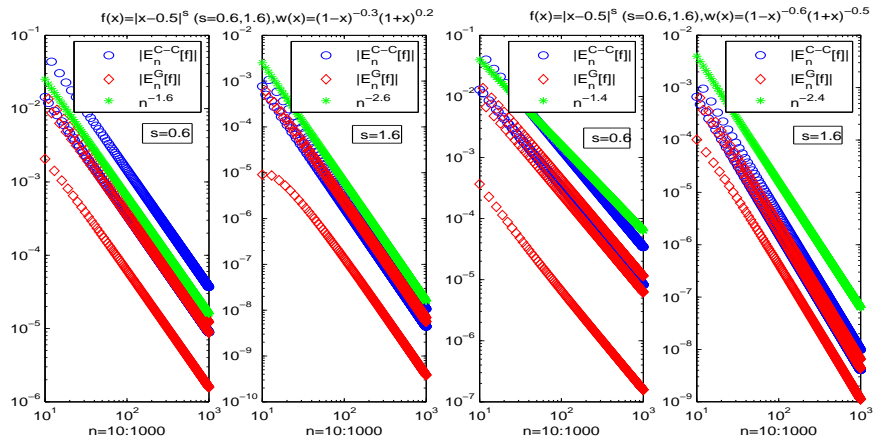}}
\caption{The absolute errors for $n$-point Gauss and Clenshaw-Curtis for $f(x)=|x-0.5|^s$ ($f\in X^s$) and $w(x)=(1-x)^{\alpha}(1+x)^{\beta}$: $n=10:1000$.}
\end{figure}

\vspace{0.36cm} {\bf Acknowledgement}. The author is grateful to
Prof. Xiaojun Chen at Hong Kong Polytechnic University for her helpful comments and kind supports.

\end{document}